\documentclass[12pt]{article}

\usepackage{CJK,CJKnumb,amsmath}
\usepackage{amsfonts}
\usepackage{amsthm}
\usepackage{geometry}
\usepackage{mathrsfs}
\usepackage{amsmath}
\usepackage{amssymb}
\usepackage{amsfonts}
\usepackage[percent]{overpic}
\usepackage{bm}
\usepackage[all]{xy}
\usepackage{graphicx}
\usepackage{subfigure}
\usepackage{latexsym}
\usepackage[colorlinks, linkcolor=blue, anchorcolor=blue, citecolor=blue]{hyperref}
\usepackage{epigraph}
\usepackage{hyperref}
\usepackage{fancyhdr}
\usepackage{comment}
\usepackage{tikz}
\usetikzlibrary{arrows.meta}
\usetikzlibrary{shapes.geometric}
\usetikzlibrary{arrows.meta,decorations.pathreplacing}

\usepackage{mathtools}
\mathtoolsset{showonlyrefs}

\numberwithin{equation}{section}
\usepackage{color}

\newtheorem{theorem}{Theorem}

\newtheorem{que}{Question}

\newtheorem{lemma}{Lemma}[section]
\theoremstyle{remark}
\newtheorem{remark}{Remark}

\newtheorem*{ack}{Acknowledgement}

\def\re{\operatorname{Re}}

\def\meas{\operatorname{meas}}
\def\diam{\operatorname{diam}}
\def\dist{\operatorname{dist}}

\def\c{\operatorname{\mathbb C}}

\def\p{\operatorname{\mathcal{P}}}

\begin{document}
	\title{Dynamical (non-)recurrence of escaping exponential maps}
	\author{Weiwei Cui, Jiaxing Huang and Jun Wang}
	\date{}
	\maketitle
	
\begin{abstract}
We show that exponential maps could have complicated measurable behaviours by	 giving two examples in the exponential family both with singular value escaping to infinity ``slowly", one of which is recurrent (in the sense that every positive measure set will return to itself infinitely many times), while the other is non-recurrent. This also complements earlier results of Lyubich, Rees, Urba\'nski-Zdunik, and Hemke. We also give an example of exponential maps with singular value escaping to infinity arbitrarily slow, which is a parametric analog of a result of Rempe.
\end{abstract}
	
\section{Introduction and main results}

The exponential family
\[\left\{f_{\lambda}(z)=\lambda e^z:\,\lambda\in\c\setminus\{0\}\right\}\]
is one of the most well-studied families in transcendental dynamics. It has only one singular value $0$. The singular behavior usually determines the global dynamics of the maps. In this note, we continue our study of measurable dynamics of exponential maps, particularly focusing on escaping exponential maps.  Recall that an exponential map $f_{\lambda}$ is called \emph{escaping} if $f_{\lambda}^n(0)\to\infty$ as $n\to\infty.$ In this case, we also say that $\lambda$ is an escaping parameter.

A directly related concept in this paper is dynamical recurrence. Say that $f_\lambda$ is \emph{dynamically non-recurrent} if there exists a set $A\subset\mathbb C$ of positive Lebesgue measure such that
\[
   A\cap f_\lambda^n(A)=\emptyset \qquad\text{for every }\,n\geq 1.
\]
Our main goal is to investigate the relationship between the escape behavior of the singular value and dynamical recurrence by constructing several examples.

Lyubich and independently Rees proved that $f_1(z)=e^z$ is dynamically non-recurrent \cite{lyubich15, rees5} .  More precisely, for Lebesgue almost every point in the plane, its $\omega$-limit set is the postsingular set of $f_1$, together with $\infty$.  This result was extended by Urba\'nski and Zdunik to escaping exponential maps satisfying a sector condition \cite[Corollary 5.4]{urbanski3}, and by Hemke to maps satisfying a parabola condition (see, e.g. \cite{hemke1}). In all of these results we mentioned, the singular orbit escapes to infinity at iterated exponential speeds. We first present some examples of ``slowly" escaping exponential maps which are also dynamically non-recurrent.

\begin{theorem}\label{thm1}
There exists a ``slowly" escaping exponential map which is dynamically non-recurrent.
\end{theorem}

\smallskip
	
The aforementioned results and Theorem \ref{thm1} suggest dynamical non-recurrence of some escaping exponential maps. It is therefore natural to ask whether every escaping exponential map enjoys this property. We show that this is not the case by constructing an example with singular value escaping to $\infty$ ``slowly", while the map itself is dynamically recurrent. Therefore, compared with fast escaping maps, “slowly” escaping maps do not usually have uniform measurable dynamics.

\begin{theorem}\label{thm2}
There exists an escaping exponential map which is dynamically recurrent.
\end{theorem}

\begin{remark}
Here we use the term ``dynamical recurrence" in order to make a distinction from the so-called \emph{non-recurrent} exponential map. By definition, an exponential map is non-recurrent (or singularly non-recurrent, to be more precise) if the \emph{singular value} is non-recurrent; see, for instance, \cite{aspenberg-cui}.
\end{remark}
	
\begin{remark}
Non-recurrent exponential maps may be dynamically recurrent. For instance, a result of Bock tells that post-singularly finite exponential maps (which are clearly non-recurrent) are dynamically recurrent \cite{bock}. It seems difficult to formulate a sufficient and necessary condition on the singular orbit (even in the non-recurrent setting) for dynamical recurrence or dynamical non-recurrence.
\end{remark}

	The method used in the construction of the above theorem also gives a parameter version of a result of Rempe concerning the existence of arbitrary escaping speed of escaping points in the Julia set; see, \cite[Theorem 1.4]{rempe12} and also \cite{rippon4}.
	
\begin{theorem}
Let $(r_n)$ be a sequence of positive real numbers so that $r_n\to\infty$ and $r_{n+1}\leq e^{r_n}+c$ for some $c>0$. Then there exists an escaping parameter $\lambda$ and some $n_0\in\mathbb{N}$ so that $|\re f_{\lambda}^{n}(0)-r_n|\leq \pi$ for $n\geq n_0$.
\end{theorem}

In other words, there are parameters with arbitrarily slowly escaping speed of the singular value.

	We say a few words about constructions in the above theorems. Theorem \ref{thm2} uses an example constructed in \cite{cui-wang}, where such function was constructed by approximation of post-singularly finite exponential maps and turns out to be ergodic. The construction in the Theorem \ref{thm1} uses different ideas. More precisely, the singular orbit of this example is obtained by welding ``initial" pieces of singular orbits of fast escaping exponential maps while keeping a ``slow" escaping speed. By fast escaping here we mean that these exponential maps can be chosen such that the singular value escapes to $\infty$ within a sector or a parabola (i.e., they satisfy the sector or parabola condition mentioned above).

\section{Proof of the Theorem \ref{thm1}: Slow escaping and dynamical non-recurrence}

Let $f_{\lambda}$ be an exponential map. The \emph{post-singular set} of $f_{\lambda}$ is defined as
\[\p(f_{\lambda}):=\overline{\bigcup_{j\geq 0}f_{\lambda}^{j}(0)}. \]
An exponential map $f_{\lambda}$ (or a parameter $\lambda$) is called \emph{non-recurrent} if $0\not\in\omega(0)$, and \emph{escaping} if $f_{\lambda}^{n}(0)\to\infty$ as $n\to\infty$. The purpose of this section is to find, through perturbations in the parameter space, an escaping exponential map  which is dynamically non-recurrent, but with the singular value tending to infinity under iterates very “slowly”.

The construction can actually be done by perturbing any non-recurrent parameters. For simplicity, we shall start with the escaping parameter $\lambda_0=1$, i.e., $f_{1}(z)=e^z$. Define, for each $n\in\mathbb{N}$,
\[
\xi_n(\lambda):=f_{\lambda}^{n}(0),
\]
which build a connection between the parameter and phase space. Consider, for sufficiently small $r_0>0$ and the parameter disk $D(\lambda_0, r_0)$. By perturbing in this parameter disk, we find a sequence of escaping parameters $\lambda_k$ with singular values tending to infinity sufficiently fast and satisfying the sector or parabola condition mentioned in the introduction, such that $\lambda_k\to\lambda$ with $\lambda$ being ``slowly" escaping. The most important feature of maps $f_{\lambda_k}$ is that the singular orbit of $f_{\lambda_k}$ is trapped in some left half-plane for quite a long time and then is released for free so that it can go to infinity at an iterated exponential speed.

The following lemma \cite[Lemmas 3.6 and 3.7]{aspenberg-cui}, establishing distortion and expansion properties of $\xi_n$, is useful in our construction and will be repeatedly used. It is stated for any non-recurrent parameters and thus applies particularly to the escaping ones. We first need some notations.

Note that any non-recurrent exponential map $f_{\lambda_0}$ is expanding on their post-singular set \cite[Corollary A]{benini3}, thus there is a neighborhood of $\p(f_{\lambda_0})$ such that $f_{\lambda_0}$ has expansion in this neighborhood. Moreover, there is a holomorphic motion $h: D(\lambda_0, r)\times \p(f_{\lambda_0})\to\c$ so that 
$$h_{\lambda}\circ f_{\lambda_0}=f_{\lambda}\circ h_{\lambda}$$
for $\lambda\in D(\lambda_0, r)$ and for $z\in \p(f_{\lambda_0})$; see \cite[Lemma 2.4]{aspenberg-cui}. Define  $\mu_j(\lambda)=h_{\lambda}(\xi_j(\lambda_0))$. By choosing $r$ small enough, we can take some number $\delta>0$ such that the set of points $z$ with $\dist(z, h_{\lambda}(\p(f_{\lambda_0})))<10\delta$ is contained in this neighborhood of expansion for all $\lambda\in D(\lambda_0, r)$.

\begin{lemma}\label{dislarge}
Let $\lambda$ be a non-recurrent parameter. For any $\varepsilon>0$ there exist $\delta>0$ and $r>0$ sufficiently small such that for any $\lambda_1,\lambda_2\in D(\lambda, r)$, if $|\xi_j(\lambda_i)-\mu_j(\lambda_i)|\leq \delta$ for $i=1,2$ and for $j\leq n$, then
\[
\left|\frac{\xi'_n(\lambda_1)}{\xi'_n(\lambda_2)}-1 \right|<\varepsilon.
\]
Moreover, there exists $S>0$ such that for all $r>0$ small enough one can find $N\in\mathbb{N}$ such that $\xi_{N}(D(\lambda, r))\supset D(\xi_{N}(\lambda),S/4)$ and $\diam\xi_{N}(D(\lambda, r))\geq S$.
\end{lemma}

\subsection{Finding escaping parameter by perturbations}\label{sec:inductive}
As mentioned before, we start with $\lambda_0=1$ and the parameter disk $D(\lambda_0, r_0)$ for some small $r_0>0$. By Lemma \ref{dislarge}, given a number $S>0$, there exists $N_0\in\mathbb{N}$ depending on $r_0$ such that $\xi_{N_0}(D(\lambda_0, r_0))$ with $\diam\xi_{N_0}(D(\lambda_0, r_0))\geq S$ contains the disk $D(\xi_{N_0}(\lambda_0), S/4)$ and the function $\xi_{N_0}$ has bounded distortion.
	
For $z\in\c$ and $r>0$, let $Q(z, r)$ denote the square centered at $z$ of sidelength $r$. Note that $\xi_{n}(\lambda_0)$ lies on the positive real axis for all $n$. Therefore, by choosing $r>0$ sufficiently small we may achieve that 
$$\inf_{z\in D(\xi_{N_0}(\lambda_0), S/4)}|f_{\lambda_0}'(z)|\geq C:=16\pi/S.$$
This implies that $f_{\lambda_0}(D(\xi_{N_0}(\lambda_0), S/4))\supset Q(\xi_{N_0+1}(\lambda_0), \pi)=:Q_0$.
	
Put
$$x_0=\re \xi_{N_0+1}(\lambda_0).$$
Note that
$$f_{\lambda_0}(Q_0)=\left\{z\in\mathbb{C}:e^{x_0-\pi/2}<|z|<e^{x_0+\pi/2},\, \re z>0\right\}=:A_0,$$
which is a large half annulus in the right half-plane. Since $x_0$ can be taken sufficiently large (by taking $r_0$ sufficiently small), one can find finitely many $2k\pi i$ translates of $Q_0$ contained in $A_0$. Suppose that there are $M$ such squares and we denoted them by $Q_1,Q_2,\dots, Q_M$. Then it is clear that $f_{\lambda_0}(Q_j)=A_0$ holds for all $1\leq j\leq M$.
	
Now we have the freedom of choosing a parameter in $D(\lambda_0, r_0)$ such that the singular orbit of this parameter first follows that of $\lambda_0$ for $N_0+1$ iterates and then jumps among the squares $Q_j's$ in a ``random" way for staying there as many iterates as we want. More precisely, suppose that the total time spent in the squares is $p_0$, 
and let $R_1, R_2, \cdots,  R_{p_0}$ be the squares in $\{Q_1, \dots, Q_M\}$ for which we choose and the indices are arranged in the order they are visited. (Note that it is possible that $p_0>M$ so that some squares will be visited several times). Then we define, for $1\leq j\leq p_0$
$$Q'_j:=\xi_{N_0+j+1}^{-1}(R_j).$$
By construction, $\overline{Q'_{j+1}}\subset Q'_{j}$. So we see that if
$$\lambda\in \bigcap_{j=1}^{p_0}Q'_j,$$
then the singular orbit of $f_{\lambda}$, after first $N_0+1$ iterates staying closely with that of $f_{\lambda_0}$, will first get into $R_1$, and then $R_2$, and so forth.

Therefore, we have chosen many parameters whose singular orbits stay in some compact sets as long as we wish (by taking $p_0$ as large as possible). In the following, we need to select a parameter from them such that its singular orbit escapes exponentially fast afterward.  This is relatively easy. A similar argument can be found in \cite{qiu2}. Note that $f_{\lambda_0}(R_{p_0})=A_0$. Therefore, one can find a square $R_{p_0+1}\subset A_0$ centred at $w_1\in\mathbb{R}$ of sidelength $\pi$ such that $\min_{z\in R_{p_0+1}}\re(z)\geq e^{x_0}$. 
While $f_{\lambda_0}(R_{p_0+1})$ is again a half annulus, we see that one can find a square $R_{p_0+2}$ of sidelength $\pi$ centered at $w_2\in\mathbb{R}$ in this annulus such that $w_2>w_1$ and  $\min_{z\in R_{p_0+2}}\re(z)\geq e^{e^{x_0}}$. 
Repeating this procedure over and over again, one can find a parameter $\lambda_1\in D(\lambda_0, r_0)$ such that $f_{\lambda_1}^{n}(0)$ follows closely $f_{\lambda_0}^{n}(0)$ for $n\leq N_0+1$ 
and then $f_{\lambda_1}^{N_0+1+j}(0)\in R_{j}$ for $j\leq p_0$, 
and then $f_{\lambda_1}^{n}(0)\in R_{n-N_0}$ for all $n\geq N_0+p_0+2$, where $R_{p_0+j}, j\geq 1$ is the square centered at $w_j\in\mathbb{R}$ such that $f_{\lambda_0}(R_{p_0+j})$ is a half annulus containing $R_{p_0+j+1}$ and $\min_{z\in R_{p_0+j}}\re(z)\geq f_{\lambda_0}^{j}(x_0)$.

\smallskip

By construction, $f_{\lambda_1}^n(0)$ follows closely the positive real axis for all large $n$ and thus the singular value $0$ will escape to $\infty$ under iterates of $f_{\lambda_1}$ in some sector or parabola containing the positive real axis. The following lemma says that $f_{\lambda_1}$ is dynamically non-recurrent and actually gives stronger measurable behavior for such functions.
	
\medskip

\begin{lemma}\label{typical}
Let $\lambda$ be an escaping parameter. If 
\begin{equation}\label{parabola}
\re f_{\lambda}^{n}(0)\geq |f_{\lambda}^{n}(0)|^{\delta}
\end{equation}
for some $\delta>0$ and for all but finitely many $n$, then $\omega(z)=\p(f_{\lambda})\cup\{\infty\}$ for almost every $z\in\c$.
\end{lemma}

\begin{remark}
This lemma was proved in the thesis of Hemke, see \cite[Corollary 6.1]{hemke3}, which is also a special case of \cite[Theorem 1.3]{hemke2}. We refer \eqref{parabola} as the ``parabola condition". In \cite[Corollary 5.4]{urbanski3}, this lemma was proved under the stronger condition that the singular orbit escapes in some sector.
\end{remark}

The above process can be repeated since $\lambda_1$ is an escaping parameter and thus  Lemma \ref{dislarge} can be used again.

Starting with the escaping parameter $\lambda_1$ we will construct another parameter $\lambda_2\in D(\lambda_1, r_1)$ for some $r_1<r_0$ whose singular orbit stays closely with that of $\lambda_1$ and then stay in some compact set again for quite a long time and then goes to infinity sufficiently fast under iterates. The construction is same to the above, we omit details.

Suppose that a sequence of parameters can be constructed which tends to some limiting parameter. We will show that the limiting parameter is also escaping, but the corresponding map is dynamically non-recurrent. An ingredient in this proof is a quantitative version of a result by Hemke \cite[Theorem 3.1]{hemke2}. Similar idea also appeared in the work of Rees \cite{rees5} and Lyubich \cite{lyubich15}, respectively.

\subsection{Limiting maps and dynamical non-recurrence}
For a parameter $\lambda$, some $R>0$ and $\eta>0$, define
\[V_{\lambda, R,\eta}=(\c\setminus \overline{D(0,R)})\ \cup  \left( \bigcup_{z\in\p(f_{\lambda})}D(z,\eta)\right).\]
For some number $0<c<1$ we say that a square $S$ is $\delta$-good for $\delta>0$ if
\[ \frac{c}{8}\left(\inf_{z\in c^{-1} S}|z|\right)^{-\delta}   \leq\diam S\leq \frac{c}{2}\left(\sup_{z\in c^{-1} S}|z|\right)^{-\delta}.\]
Here $c^{-1}S$ means a larger square $S'$ with the same center as that of $S$, but its sidelength is $c^{-1}$ times that of $S$. For fixed $c$, $S'\setminus\overline{S}$ has definite modulus and thus if a univalent function is defined on $S'$, Koebe's distortion applies on $S$. 

This definition uses that of Hemke. But the precise number $\delta$ will not concern us as we only use a finite version of Hemke's theorem, in the proof of which these squares are used to cover the whole tracts. So we fix $\delta$ throughout the rest of proof and only say \emph{good squares}. Let $\meas(A)$ denote the Lebesgue measure of $A\subset\c$.

\begin{lemma}\label{hemkequa}
Let $\lambda$ satisfy the parabola condition (\ref{parabola}). Let $R>0$, $\eta>0$ and $\varepsilon>0$ be given. Then there exists $M>R$ and a good square $S\subset \c\setminus D(0,M)$ such that the following holds: there exist finitely many, say $m$, disjoint subsets $V_i\subset S$, $p_i\in\mathbb{N}$ and good squares $S_i$ for $1\leq i\leq m$ such that
\begin{itemize}
\item[$(i)$] $f^{p_i}_{\lambda}: V_i\to S_i$ is conformal and onto,
\item[$(ii)$] $f^{j}_{\lambda}(V_i)\subset V_{\lambda, R,\eta}$ for $1\leq j\leq p_i$, and
\item[$(iii)$] $\meas \left(S\setminus\cup_{i=1}^{m} V_i\right)\leq \varepsilon \meas S$.
\end{itemize}
\end{lemma}

\begin{remark} We will not prove this, as this is a special version of Theorem 3.1 in \cite{hemke2}. But let us sketch the basic idea. Cover some right half-plane with disjoint good squares. For the square $S$, its image $f_{\lambda}(S)$ is either very large and thus cover many good squares or mapped very close to the singular value $0$. In the first case, pulling back these good squares to $S$ gives finitely many disjoint subsets of $S$ which satisfy $(i)$ and $(ii)$. To make $(iii)$ work for fixed $\varepsilon$, we continue to iterate the rest part of $S$. After some iterates, it is either reaching a large size which then cover some good squares or mapped to some left-half plane (which is the second case). For the second case, finitely many iterates will follow closely the singular orbit, but will reach to some large size due to the expansion on the singular orbit which is ensured by the parabola condition. As long as large size is realized, the above dichotomy will occur again. Continuing the above process until $(iii)$ is satisfied. Since $\varepsilon$ is prescribed, only finitely many disjoint subsets are needed.
\end{remark}

Fix $R>0$ and choose a sequence $(\theta_k)_{k\geq0}$ such that
\begin{equation}\label{eq:summable-losses}
    0<\theta_k<\frac12,
    \qquad
    \sum_{k=0}^{\infty}\theta_k<\infty.
\end{equation}

Suppose now that $\lambda_k$ has been constructed. Choose $\eta_k>0$, with
\[D_{R,k}:=D(0,R)\cap \bigcup_{z\in\p(f_{\lambda_k})}D(z,\eta_k)\]
we have
\[\meas\left(D_{R,k}\right)\leq \frac{1}{2^{k+2}}\meas D(0,R).\]
So the set
\[A:=D(0,R)\setminus \bigcup_{k=0}^{\infty}D_{R,k}\]
has positive measure.

Let now $S_0$ be a fixed good square for $\lambda_0=1$. 
Applying Lemma \ref{hemkequa} to the parameter $\lambda_k$, 
we obtain a finite family $\mathcal{A}_k=\{V_{k, i}: i=1,\dots, m_k\},$ and $p_{k,i}\in\mathbb{N}$ such that $$f_{\lambda_k}^{p_{k,i}}: V_{k, i}\to S_{k,i}$$ is conformal and onto where $S_{k,i}$ is a good square. Moreover, by Lemma \ref{hemkequa}(ii), first $p_{k,i}$ iterates of $V_{k, i}$ do not intersect with $A$.

Fix $\eta_k>0$. Now applying Lemma \ref{hemkequa} to each of $S_{k,i}$ for all $1\leq i\leq m_k$. If $V\subset S_{k,i}$ is the domain ensured by this lemma, then $f_{\lambda_k}^{-p_{k,i}}(V)$ is a subset of $V_{k,i}$. By the definition of good squares and Lemma \ref{hemkequa}(iii), we can make sure that the pullback of all $V$'s contained in $S_{k,i}$ has measure greater than $(1-\theta_k)\meas V_{k,i}$.

Now taking $r_k>0$ sufficiently small with $\overline{D}(\lambda_k,r_k)\subset D(\lambda_{k-1},r_{k-1})$ (i.e., making perturbations small enough) we may achieve that the above estimates hold for all parameters for $\lambda\in D(\lambda_k,r_k)$. Then we perform similar constructions to obtain a parameter $\lambda_{k+1}$ in $D(\lambda_k,r_k)$ to get similar estimates with $\theta_k$ replaced by $\theta_{k+1}$.

Induction of the above procedure gives a sequence $\{\lambda_k\}_{k\geq 0}$ of escaping parameters and $r_k$ so that $\overline{D}(\lambda_k, r_k)\subset D(\lambda_{k-1}, r_{k-1})$. Take
$$\lambda_{\infty}\in\bigcap_{k=0}^{\infty}\overline{D}(\lambda_k,r_k).$$
Then $\lambda_{\infty}$ is a ``slowly" escaping parameter in the sense that the singular orbit stays occasionally in a sequence of compact sets for as many iterates as one might wish, while \(\re f_{\lambda_\infty}^n(0)\to+\infty.\)

Put
\[\widehat{\mathcal{A}_k}:=\bigcup_{V\in \mathcal{A}_k} \overline{V}.\] It follows that $$\widehat{\mathcal{A}_0}=\{S_0\}, \qquad\widehat{\mathcal{A}_{k+1}}\subset \widehat{\mathcal{A}_k}$$ and 
$$\meas(\widehat{\mathcal{A}_{k+1}})\geq (1-\theta_k)\meas(\widehat{\mathcal{A}_k})\geq\prod_{i=0}^k(1-\theta_i)\meas(S_0).$$
Hence, set \[B:=\bigcap_k\, \widehat{\mathcal{A}_k},\] we have $$\meas(B)=\lim_{k\to\infty}\meas(\widehat{\mathcal{A}_{k}}).$$
Then by our choice of $\theta_k$, we obtain that 
\[\meas B \ge \prod_{i=0}^{\infty}(1-\theta_i)\meas(S_0) >0.\]
By the above construction of $A$ and $B$, we see that
\begin{equation}\label{eq:positive-area-avoidance}
    f_{\lambda_\infty}^n(B)\cap A=\emptyset
    \qquad\text{for every }n\geq0.
\end{equation}

Therefore, we have obtained two sets $A$ and $B$, both of which have positive measure. The following result of Bock \cite{bock} tells us the asymptotic behaviours of typical points for our limiting parameter.

\begin{lemma}\label{lem:bock-dichotomy}
For any non-constant function $f$ meromorphic in the complex plane, one of the following two cases must occur:
\begin{itemize}
\item[$(a)$] The Julia set of $f$ is either the whole complex plane, and for all $A\subset\c$ of positive measure, all $m\in\mathbb{N}$ and almost all $z$ there are infinitely many $n$ with $f^{mn}(z)\in A$;

\item[$(b)$] almost every forward orbit in the Julia set accumulates only on the post-singular set.
\end{itemize}
\end{lemma}

We now apply Lemma~\ref{lem:bock-dichotomy} to the exponential map
\(f:=f_{\lambda_\infty}\). In this case,
Lemma~\ref{lem:bock-dichotomy}(a) cannot hold, as we have positive measurable sets $A$ and $B$ with the property \eqref{eq:positive-area-avoidance}. Therefore,
Lemma~\ref{lem:bock-dichotomy}\textup{(b)} must hold. Since $\lambda_{\infty}$ is escaping, we know that its Julia set $\mathcal J(f)$ is the whole
plane. Hence, we have
\begin{equation}\label{eq:ErE}
        \omega(z)\subset \p(f)\cup\{\infty\}    
\ \text{for almost every}\ z\in\mathcal J(f)=\c.
\end{equation}
Choose a closed disk
\[
    E\subset\mathbb C\setminus \p(f)
\]
of positive area. Then, almost every point
of $E$ returns to $E$ only finitely many times. Otherwise, the orbit
would have an accumulation point in the compact set $E$, contradicting
\eqref{eq:ErE}.

For $m\geq0$, define
\[
    E_m:=
    \left\{
        z\in E:
        m=\max\{n\geq0:f^n(z)\in E\}
    \right\}.
\]
The sets $E_m$ cover $E$ up to a set of measure zero, so some $E_m$ has positive
area.  Put
\[
    F:=f^m(E_m).
\]
Then, $\operatorname{meas}(F)>0$ and $F\subset E$. If
$w=f^m(z)\in F$ and $n\geq1$, the definition of $E_m$ gives
\[
    f^n(w)=f^{m+n}(z)\notin E.
\]
Since $F\subset E$, it follows that
\[
    F\cap f^n(F)=\emptyset
    \qquad\text{for every }n\geq1.
\]
Thus,
$f$ is dynamically non-recurrent. Together with the
slow escape property of $\lambda_\infty$, this completes the proof of
Theorem~\ref{thm1}.

\section{Proof of the Theorem \ref{thm2}: Slow escaping and dynamical recurrence}
	
The construction of an escaping exponential map in \cite[Theorem 1.1]{cui-wang} provides post-singularly finite parameters $\lambda_k$ converging to an escaping parameter $\lambda$.  The map $f_\lambda$ is ergodic.  We explain why the same construction yields dynamical recurrence and give a sketch of proof only.

\medskip

As in \cite[pages 3997-3998]{cui-wang}, for the constructed post-singularly finite parameter $\lambda_k$, taking $\eta_k>0$ small enough such that $\bigcup_{z\in\p(f_{\lambda_k})}D(z,\eta_k)$ is of Lebesgue measure $\leq \delta_k$. Take $R>0$. Then for any point $w\in D(0,R)\setminus \cup_{z\in\p(f_{\lambda_k})}D(z,\eta_k)$, one can find $M_k$ so that the preimages of $D(z,\eta_k/4)$ under $f_{\lambda_k}^{j}$ for $j\leq M_k$ cover $D(0,R)$, except for a set of Lebesgue measure at most $\delta_k$.

By taking $r_k$ sufficiently small, the Lebesgue measure of 
\[D(0,R)\setminus \bigcup_{j=0}^{M_k}f_{\lambda'}^{-j}D(z,\eta_k/4)\]
is less than $2\delta_k$ for all $\lambda'\in D(\lambda_k, r_k)$ and $f_{\lambda'}^{-j}$ has bounded distortion which inherits from that of $f_{\lambda_k}^{-j}$. Since components of $f_{\lambda_k}^{-j}D(z,\eta_k/4)$ intersecting with $D(0,R)$ has diameter at most $\varepsilon_k$ (by bounded distortion) and $\varepsilon_k\to 0$ as $k\to\infty$ (by expansion of $f_{\lambda_k}$ in the hyperbolic metric of $\c\setminus \p(f_{\lambda_k})$), components of $f_{\lambda'}^{-j}D(z,\eta_k/4)$ have diameter at most $2\varepsilon_k$.

Now suppose that $f_{\lambda}$ is dynamically non-recurrent. Then there exists a positive measure set $A$ such that there is a positive measure subset $B\subset A$ satisfying $f_{\lambda}^n(B)\cap B=\emptyset$, for all large $n$. Suppose that there exists $M>0$ such that $f_{\lambda}^n(B)\cap B=\emptyset$ for all $n\geq M$. It is enough to assume that $B\cap D(0,R)$ has positive measure. Otherwise, the following argument works for some preimage or image of $B$ under certain iterates. To see this, note first that it was proved, in \cite{cui-wang}, that $f_{\lambda}$ is ergodic and thus $\cup_{n\in\mathbb{Z}}f_{\lambda}^{n}(B)$ must have full Lebesgue measure. Thus for $R$ as above, there exists $n_0\in\mathbb{Z}$ so that $f_{\lambda}^{n_0}(B)\cap D(0,R)$ has positive Lebesgue measure. Then it is enough to work with $f_{\lambda}^{n_0}(B)\cap D(0,R)$.

\smallskip

With these information, the construction in \cite{cui-wang} gives a slowly escaping parameter $\lambda$. Since $\delta_k\to0$ as $k\to\infty$, we take a Lebesgue density point of $\alpha_1\in B\cap D(0,R)$ but $\alpha_1\not\in \bigcup_{z\in\p(f_{\lambda_k})}D(z,\eta_k)$. Therefore, the above discussion on the Lebesgue measure of preimages of $D(z,\eta_k/4)$ applies with $z=\alpha_1$, and thus  one can take another Lebesgue density point $\alpha_2\in B\cap D(0,R)$ so that $\alpha_2$ belongs to some component of $f_{\lambda_k}^{-p_{k_1}}D(\alpha_1, \eta_k/4)$ for some $p_{k_1}\leq M_k$. The above procedure applies to $\alpha_2$, there can be found a Lebesgue density point $\alpha_3\in B\cap D(0,R)$ so that $\alpha_3$ lies in some component of $f_{\lambda_k}^{-p_{k_2}}D(\alpha_1, \eta_k/4)$, where $p_{k_2}\leq M_k$. Repeating this for a finite number of times, say $m$, one can find a sequence of $p_{k_1}, p_{k_2},\dots, p_{k_m}$ such that $\sum_{j=1}^{m}p_{k_j}>M$. However, note that during each of these procedures, $\alpha_j$ is a Lebesgue density point and $f_{\lambda}$ has bounded distortion, this implies that a definite portion of $B$ returns to $B$ after $M$ iterates of $f_{\lambda}$. This is a contradiction.

\medskip

Motivated by these examples from Theorem \ref{thm1} and Theorem \ref{thm2} and also previous results of Urba\'nski-Zdunik and Hemke, one could naturally ask the following question.

\begin{que}
What is the optimal escaping speed of the singular value for an escaping exponential map $f_{\lambda}$ to be dynamically non-recurrent?
\end{que}

A related question concerning non-ergodicity of escaping exponential maps was posed in \cite[Question 5.12]{bergweiler25}. Hemke proved that the parabola condition is enough to ensure dynamical non-recurrence \cite[Corollary 6.1]{hemke3}.

The examples constructed in Theorem \ref{thm2} are both ergodic and recurrent. The former is the main result of \cite{cui-wang}. The following result tells that a non-ergodic escaping exponential map must be dynamically non-recurrent.

\begin{theorem}
Let $f_{\lambda}$ be escaping. Suppose that $f_{\lambda}$ is not ergodic. Then $f_{\lambda}$ is dynamically non-recurrent.
\end{theorem}

\begin{proof}
This is essentially already in the proof of Theorem \ref{thm1}. By the result of Bock \cite{bock}, which is Lemma \ref{lem:bock-dichotomy}, $f:=f_{\lambda}$ satisfies $(b)$. Thus Lebesgue almost every point in the plane will accumulate only on the post-singular set $\p(f)$.

Since $f$ is escaping, let $D$ be a disk disjoint from $\p(f)$. So almost every point will not return to $D$ infinitely often. So for $z\in D$ except for a set of zero measure, we define
\[\alpha(z):=\max\{n\in\mathbb{N}:\, f^n(z)\in D\}.\]
Put
\[D_n=\{z\in D:\, \alpha(z)=n\}.\]
Then the sets $D_n$ cover $D$ up to the set of measure zero. This also implies that there exists some $m$ such that $D_m$ has positive measure. Put $D'=f^m(D_m)$. Then $D'$ is a subset of $D$ and has positive measure. By construction, for all $n\geq 1$,
\[D'\cap f^n(D')=\emptyset.\]
Therefore, $f$ is dynamically non-recurrent.
\end{proof}

Thus the following interesting question remains open.

\begin{que}
Is there an escaping exponential map which is ergodic but dynamically non-recurrent?
\end{que}

\section{Arbitrarily slowly escaping singular orbits}

In this section, we sketch how a parameter with arbitrarily slowly escaping singular orbits is obtained. This construction is similar to the one in the previous section, so we only sketch the construction.

We again start with the standard exponential map $f_{1}(z)=e^z$ (this is not necessary but only for the sake of convenience), and set $\lambda_0=1$. By Lemma \ref{dislarge}, for a given $r>0$, there exists $N_0\in\mathbb{N}$ such that $\xi_{N_0}(D(\lambda_0, r))$ has reached the large scale $S$ and $\xi_{N_0}$ has bounded distortion. Again, by taking $r$ small enough, we will have $\re\xi_{N_0}(\lambda_0)$ as large as we wish, so that $f_{\lambda_0}(D(\xi_{N_0}(\lambda_0), S/4))\supset D(\xi_{N_0+1}(\lambda_0), \pi/2)$.

Recall that we have a prescribed sequence $(r_n)$ such that $r_n\to\infty$ and $r_{n+1}\leq e^{r_n}+c$ for some $c>0$. We need to compare $\re\xi_{N_0+1}(\lambda_0)$ with $r_{N_0+1}$.

If $|\re\xi_{N_0+1}(\lambda_0)-r_{N_0+1}|\leq \pi/2$, then we are in good situation. In the following steps, we only need to consider squares whose centers have real parts equal to $r_n$ for $n>N_0+1$. More precisely, for a square centered on the vertical line $\{\re z=r_n\}$, the image under $f_{\lambda_0}$ is a half annulus in the right half-plane. The condition $r_{n+1}\leq e^{r_n}+c$ implies that this half annulus will intersect with the vertical line of the real part $r_{n+1}$ and thus we can choose a square whose center lies on $\{\re z=r_{n+1}\}$. Repeating this procedure gives us a parameter $\lambda$ with singular orbits tending to infinity but the escaping speed controlled by the sequence $(r_n)$. Then theorem holds by putting $n_0=N_0+1$.

If $\re\xi_{N_0+1}(\lambda_0)>r_{N_0+1}$, we can choose $N>N_0$. Then it takes at most $N-N_0$ iterates for a square at $\re\xi_{N_0+1}(\lambda_0)$ to cover a square with center at $\re z=r_N$. The rest of the construct is similar to the above. We put $n_0=N$.

If $\re\xi_{N_0+1}(\lambda_0)<r_{N_0+1}$, then we can continue to iterate $\xi_{N_0+1}(\lambda_0)$ several times, say $p$, such that $\re\xi_{N_0+1+p}(\lambda_0)>r_{N_0+1}$. Then we are at the second situation.

In any case, we can have an arbitrarily slowly escaping singular orbit.

\medskip

\begin{ack}
Cui was partially supported by NSFC (No. 12401105 and 12522108), Qingdao NSF (No. 24-4-4-zrjj-8-jch) and Shandong Provincial Natural Science Fund for Excellent Young Scientists Program (Overseas) (No. 2025HWYQ-021). Huang was partially supported by the NSFC (No.12201420 and 12231013) and the NSF of Guangdong Province (No. 2026A1515010851).  Wang was partially supported by the NSFC (No.12471072) and National Key R\&D Program of China Grant (No.2021YFA1003200).
\end{ack}

\bigskip

\noindent {\bf Weiwei Cui}\\
 Research Center for Mathematics and Interdisciplinary Sciences\\
Shandong University, Qingdao, 266237, China.

\smallskip

\noindent{weiwei.cui@sdu.edu.cn}

\bigskip

\noindent {\bf Jiaxing Huang}\\
School of Mathematical Sciences\\
 Shenzhen University, Guangdong, 518060, China

\smallskip

\noindent{hjxmath@szu.edu.cn}

\bigskip

\medskip

\noindent {\bf Jun Wang}\\
School of Mathematical Sciences\\
Fudan University, Shanghai, 200433, China.

\smallskip

\noindent{majwang@fudan.edu.cn}

\end{document}